\documentclass[onefignum,onetabnum]{siamart220329}
\usepackage{algorithmic,algorithm,lineno,bm}
\makeatletter
\renewcommand{\ALG@name}{\sc Algorithm}
\newsiamremark{example}{\sc Example}
\newsiamremark{remark}{\sc Remark}
\newsiamremark{assumption}{\sc Assumption}
\usepackage{chngcntr}
\usepackage{color}

\counterwithin{equation}{section}

\title{Computing local minimizers in polynomial optimization under genericity conditions}
\author{Vu Trung Hieu\thanks{Norwegian University of Science and Technology, 7034 Trondheim, Norway
        (\email{trung.h.vu@ntnu.no})}
\and Akiko Takeda\thanks{The University of Tokyo, 113-8656, Tokyo, Japan; RIKEN, 103-0027, Tokyo, Japan
        (\email{takeda@mist.i.u-tokyo.ac.jp, akiko.takeda@riken.jp})}}
\headers{Computing local minimizers in polynomial optimization}{Vu Trung Hieu and Akiko Takeda}
\ifpdf
\hypersetup{pdftitle={local minimizers in pop},pdfauthor={Hieu and Takeda}}
\usepackage{latexsym,amsmath,amsfonts,amscd,amssymb,verbatim,array,mathdots}

\DeclareMathOperator{\R}{\mathbb{R}}
\DeclareMathOperator{\C}{\mathbb{C}}
\DeclareMathOperator{\N}{\mathbb{N}}

\DeclareMathOperator{\I}{\mathcal{I}}
\DeclareMathOperator{\Sp}{\mathbb{S}}
\DeclareMathOperator{\Oo}{\mathcal{O}}

\DeclareMathOperator{\LT}{LT}
\DeclareMathOperator{\lex}{lex}
\DeclareMathOperator{\rad}{rad}

\DeclareMathOperator{\loc}{local}
\DeclareMathOperator{\glo}{global}
\DeclareMathOperator{\Mat}{\mathcal{M}}

\DeclareMathOperator{\eps}{\varepsilon}
\DeclareMathOperator{\gralom}{\mathtt{GRALOM}}
\DeclareMathOperator{\gralomplus}{\mathtt{GRALOM}^+}
\DeclareMathOperator{\grulom}{\mathtt{GRULOM}}
\DeclareMathOperator{\grulomplus}{\mathtt{GRULOM}^+}

\DeclareMathOperator{\shape}{\mathtt{shape}}

\begin{document}
\maketitle
\smallskip
\setlength\parindent{0pt}
\parskip 5pt
\setlength{\baselineskip}{12.7truept}

\begin{abstract}
In this paper, we focus on computing local minimizers of a multivariate polyno\-mial optimization problem under certain genericity conditions. By using a technique in computer algebra and the second-order optimality condition, we provide  a univariate representation for the set of local minimizers. In particular, for the unconstrained problem, i.e. the constraint set is $\R^n$, the coordinates of all local minimizers can be represented by the values of $n$ univariate polynomials at real roots of a system including a univariate polynomial equation and a univariate polynomial matrix inequality. We also develop the technique for constrained problems having equality/inequality constraints. Based on the above technique, we design symbolic algorithms to enumerate the local minimizers and provide some experimental examples based on hybrid symbolic-numerical computations. For the case that the genericity conditions fail, at the end of the paper, we propose a perturbation technique to compute approximately a global minimizer provided that the constraint set is compact.
\end{abstract}

\begin{keywords}
polynomial optimization, local minimizer, second-order optimality condition, zero-dimensional ideal, radical of an ideal, reduced Gr\"{o}bner basis, symbolic algorithm
\end{keywords}

\begin{MSCcodes}
    14P05, 13P10, 90C23
\end{MSCcodes}

\section{Introduction}\label{sec:intro}We denote by $\R$ the field of real numbers and by $x$ the $n$-tuple of variables $(x_1,\dots,x_n)$.
Let $f, h_1,\dots,h_m$ be $m+1$ polynomials in $n$ variables $x$ with real coefficients. The real algebraic set given by $h=(h_1,\dots,h_m)$ is denoted by $V_{\R}(h)$,
\[V_{\R}(h)=\{x\in{\R}^n:h_1(x)=\dots =h_m(x)=0\}.\]
A point $u\in V_{\R}(h)$ is a \textit{local minimizer} of $f$ over $V_{\R}(h)$ if there exists $\varepsilon>0$ such that $f(u)\leq f(x)$ for any $x\in V_{\R}(h)$ satisfying $\|x-u\|< \varepsilon$. Then, one says that the value  $f(u)$ is a local minimum of $f$ over  $V_{\R}(h)$.

The set of local minimizers of $f$ over $V_{\R}(h)$, denoted by $\loc(f,V_{\R}(h))$, may have infinitely many points;
for example, the set of local minimizers of $f(x_1,x_2)=(x_1-x_2)^2$ over $\R^2$ is the line $\{(x_1,x_2)\in\R^2:x_1=x_2\}$.
However, when $f,h$ are generic, according to Nie's result \cite[Theorem 1.2]{nie2014optimality}, the set has finitely many points. For such cases, in this paper, we introduce a  method to enumerate all points of the set. 

We would like to remark that, in the literature of polynomial optimization, most work is for computing global minima or infima. Computing global minima in polynomial optimization is usually based on sum of squares relaxations, for example, by  Lasserre \cite{lasserre2001global}, Parrilo \cite{parrilo2003semidefinite}, Parrilo and Sturmfels \cite{parrilo2001minimizing}, and Nie, Demmel and Sturmfels \cite{nie2006}. We refer monographs by Lasserre \cite{lasserre2009moments} or Laurent \cite{laurent2009} to this approach.
Besides, symbolic/exact algorithms can be used to compute global minima or infima such as quantifier eliminat\-ions by Basu, Pollack, and Roy in \cite{basu2013},
or without using quantifier eliminations by Safey El Din in \cite{safey2008computing}. In the literature, we can see some exact algorithms related to this topic, for example in \cite{greuet2014,berthomieu2022computing,mai2022symbolic}.

Meanwhile, local minima and minimizers are important not only in theory but also in applica\-tions \cite{nie2015hierarchy}.
For example, in modern computational chemistry, the potential energy surface is a fundamental concept to describe the energy of a chemical system as a function of its geometry.
The local minimizers correspond to stable structure, from which one may calculate energy barriers and minimum-energy paths that describe chemical reactivity, and there are various papers (e.g., \cite{he2015,morrow2021}) that attempt to find as many local minimizers as possible.
There is also a demand for methods to enumerate local minimizers in machine learning.
Some nonconvex machine learning models such as deep learning \cite{kawaguchi2016} and
low-rank matrix optimization problems based on Burer-Monteiro factorization \cite{yalcin2022} are known to have no
{\it spurious} local minimizers, which indicates that
there are no local minimizers that have larger objective values than the optimal value, i.e.,  every local minimizer is a global minimizer.
The next step in the research stream would be to enumerate local minimizers in order to determine which of multiple minimizers achieve the best prediction performance. Indeed, there is some study \cite{mehta2022} enumerating minimizers of a deep learning model.

Characterizations of local minimizers of quadratic functions over Euclidean balls, spheres or polyhedra have been investigated, for example, in \cite{martinez1994local,phu2001stability}. As shown in \cite{ahmadi2022complexity} by Ahmadi and Zhang, the problem of deciding whether a quadratic function has a local minimizer over polyhedron is NP-hard. Pardalos and Schnitger proved that it is NP-hard to test whether a given point is a local minimizer of constrained quadratic program\-ming  \cite{pardalos1988checking}.
Recently, some characterizations of local minimizers of semi-algebraic functions from the view\-point of tangencies are discussed in \cite{pham2020local} by Pham.

In \cite{nie2015hierarchy}, Nie introduces a method to compute the hierarchy of critical values and local minima of a polynomial optimization problem based on second-order optimality conditions and semi-definite relaxations. In particular, to compute a point satisfying the first and
second-order necessary optimality conditions, a sequence of
semi-definite relaxations is constructed; Moreover, under some genericity conditions,  each constructed sequence has finite convergence.

\textit{Contributions}. This work is motivated by Nie's result for the generic problems from \cite{nie2015hierarchy}. Our study aims at enumerating all local minimizers of polynomial optimization problems under genericity conditions without using semi-definite relaxations.

\begin{itemize}
    \item \textit{Unconstrained problem}. We consider the unconstrained problem given by a multivariate polynomial $f$ under genericity conditions which saying that the complex gradient variety $V_{\C}(\I_{\nabla}(f))$ has finitely many points and the Hessian matrix $\nabla_{x}^2 f$ is non-degenerate over the real gradient variety.
    In Theorem \ref{thm:local}, thanks to Shape Lemma (Lemma \ref{lm:shape0}) and the second-order optimality condition, we prove existence of univariate representations for the set of all local minimizers. Its proof allows us to design a symbolic algorithm to enumerate the local minimizers.
    Moreover, we show in Theorem~\ref{thm:global} that if the infimum of $f$ over $\R^n$ is reached, then this value can be computed though a univariate optimi\-zation problem.

    \item \textit{Constrained problem}. We also consider polynomial optimization problems with $m$ equality/inequality constraints under genericity conditions. We prove existence of univariate representations in Theorem \ref{thm:local2} and design algorithms to compute the set of local minimizers and global minimizers as well. We furthermore study the case that the genericity conditions are violated; there is a perturbation method to compute approximately a global minimizer of the considered problem.

    \item \textit{Implementation}. The algorithms are implemented in {\sc Maple}. We provide experimental examples on Rosenbrock functions with both unconstrained and constrained cases.

\end{itemize}

The remaining part of this paper consists of  five sections. Section \ref{sec:pre}
gives some definitions, notations, and auxiliary results on computational commutative algebra and optimization theory. Section \ref{sec:unconstraited}
establishes results concerning
the unconstrained case.
Results on constrained cases with equalities and inequalities are shown in Sections \ref{sec:eq} and \ref{sec:ineq}, respectively.
The last section, says Section \ref{sec:posdim}, handles the case that the involved gradient variety has positive dimension.

\section{Preliminaries}\label{sec:pre}
This section recalls basic notions and
results from computa\-tional commutative algebra and optimization theory;
Further details can be found in \cite{cox2013} and \cite{bertsekas2nd}.

\subsection{Zero-dimensional ideals and the Shape Lemma}
We denote by $\C$ the field of complex numbers that is the algebraic closure of the field $\R$. The ring of all polynomials in variables $x$ with real coefficients is denoted by $\R[x]$.
An additive subgroup $\I$ of $\R[x]$  is said
to be an \textit{ideal} of $\R[x]$ if $pg\in \I$ for any $p\in \I$ and
$g\in\R[x]$. Given $p_1,\dots,p_s $ in $ \R[x]$, we denote by
$\left\langle
p_1,\dots,p_s \right\rangle$ the ideal generated by $p_1,\dots,p_s$. If $\I$
is an ideal of $\R[x]$ then, according to Hilbert's basis theorem
(see, for example Theorem~4 in \cite[Chapter 2, \S 5]{cox2013}),
there exist $p_1,\dots,p_s \in \R[x]$ such that
$\I=\left\langle p_1,\dots,p_s \right\rangle$.

Let $\I$ be an ideal of $\R[x]$. The \textit{complex
    algebraic variety} associated to $\I$ is defined as
\[V_{\C}(\I):=\{x\in{\C}^n:g(x)=0, \forall g\in \I\}.\] The \textit{real} algebraic
variety
associated to $\I$ is $V_{\R}(\I)=V_{\C}(\I)\cap\R^n$. Recall that the ideal $\I$
is \textit{zero-dimensional} if the
cardinality of $V_{\C}(\I)$, denoted by $\# V_{\C}(\I)$, is finite. The \textit{radical} of $\I$, denoted by $\rad(\I)$, is the following set:
\[\{g\in\R[x]:g^k\in \I \text{ for some }k \in \N\}.\]
The set
$\rad(\I)$ is also an ideal that contains $\I$.
One says that the ideal $\I$ is \textit{radical} if $\I=\rad(\I)$.

The following lemma asserts that the radical of an ideal is radical; Moreover, the complex (real) varieties of $\I$ and its radical are coincident, see, for example Lemma~5 and Theorem~7  in \cite[Chapter 4, \S 2]{cox2013}:

\begin{lemma}\label{lm:radC}Let $\I$ be an ideal of $\R[x]$. Then, $\rad(\I)$ is radical. Moreover, one has $V_{\C}(\I)=V_{\C}(\rad(\I))$, hence $V_{\R}(\I)=V_{\R}(\rad(\I))$.
\end{lemma}

Let $<$ be a monomial ordering on $\R[x]$
and $\I\neq \{0\}$ be an ideal. We denote
by $\LT_<(\I)$ the set of all leading terms $\LT_<(p)$ of $p\in \I$, and  by
$\left\langle \LT_<(\I)\right\rangle $ the ideal generated by the elements
of $\LT_<(\I)$. A subset $G=\{p_1,\dots,p_s\}$ of $\I$ is said to be a \textit{Gr\"{o}bner
    basis} of $\I$ with restpect to the monomial order $<$ if
\[ \left\langle
\LT_<(p_1),\dots,\LT_<(p_s)\right\rangle=\left\langle\LT_<(\I)\right\rangle.
\]
Note that every ideal in $\R[x]$ has a Gr\"{o}bner basis. A Gr\"{o}bner
basis
$G$ is \textit{reduced}
if
the two following conditions hold: i) the leading coefficient of $p$ is $1$,
for
all $p\in G$; ii) for all $p\in G$, there are no monomials of $p$ lying in $\left\langle
\LT_<(G\setminus\{p\})\right\rangle$. Every ideal $\I$ has a unique reduced
Gr\"{o}bner basis for a given monomial ordering $<$. We refer the reader to \cite{cox2013} for more details.

Let $\I$ be a zero-dimensional and radical ideal in $\R[x]$, and
$G$ be its reduced Gr\"{o}bner basis with respect to the
lexicographical order $x_1<_{\lex} \cdots <_{\lex} x_n$. One says that $\I$ is
in \textit{shape position} if $G$ has the following shape:
\begin{equation}\label{f:sysshape}
    G=\{w,x_2-v_2,\dots,x_n-v_n\},
\end{equation}
where $w,v_2,\dots,v_n$ are polynomials in
$\R[x_1]$ and $\deg w=\# V_{\C}(\I)$.  
When the order $<_{\lex}$ is clear from the context, we omit the subscript $\lex$.

The following lemma, named Shape Lemma, gives us a criterion for the shape
position of an ideal.

\begin{lemma}{\rm{(Shape Lemma, \cite{gianni89})}}\label{lm:shape0}
    Let $<_{\lex}$ be a lexicographic monomial order in $\R[x]$. Suppose that $\I$ is a zero-dimensional and radical ideal. If the points in $V_{\C}(\I)$  have distinct
    $x_1$-coordinates, then $\I$ is in shape position as in \eqref{f:sysshape}, where $v_2,\dots,v_n$ are
    polynomials in $\R[x_1]$ of degrees at most $\# V_{\C}(\I)-1$.
\end{lemma}

\subsection{Local minimizers of unconstrained problems}
Denote by $\Mat(n\times m,\mathcal{R})$ the set of all $n\times m$-matrices with entries in a ring $\mathcal{R}$; In this paper, $\mathcal{R}$ is $\R$ or $\R[t]$, where $t$ is a single variable.

Let $Q\in\Mat(n\times n,\R)$ be a symmetric matrix. The quadratic form associated to $Q$ is given by $p(x)=x^TQx$, for $x\in\R^n$. Let $K\subset \R^n$ be a nonempty subset. One says that $Q$ is \textit{positive semi-definite} with respect to $K$, denoted by $Q \succeq_K 0$, if $p(x)\geq 0$ for all $x\in K$ and that $Q$ is \textit{positive definite} with respect to $K$, denoted by $Q \succ_K 0$, if $p(x) > 0$ for all $x\in K$ and $x\neq 0$. One sees that, if $K$ is a non-trivial cone, then checking positive semi-definiteness (resp. positive definiteness) of $Q$ with respect to $K$ is equivalent to checking positivity of the minimum of $p$ over the compact set $K\cap \Sp$, where $\Sp$ is the unit sphere in $\R^n$.
When $K=\R^n$, we say $Q$ is positive semi-definite (resp. positive definite) and write $Q \succeq 0$ (resp. $Q \succ 0$) for short. It is well-known that $Q \succeq 0$ (resp. $Q \succ 0$) if $Q$ has only non-negative (resp. positive) eigenvalues.

Suppose that the entries of matrices $Q$ and $C$ are polynomials depending on a real variable $t$, i.e. $Q\in\Mat(n\times n,\R[t])$ and $C\in\Mat(m\times n,\R[t])$. We denote by $C(t)^{\perp}$ the linear subspace $\{u\in\R^n: C(t) u=0\}$. One says that $Q(t) \succeq_{C(t)^{\perp}} 0$ and $Q(t) \succ_{C(t)^{\perp}} 0$  are univariate polynomial matrix inequalities.

Let $f$ be a polynomial in $\R[x]$. We consider the problem of computing local minimizers of $f$ over $\R^n$. A point $x\in \R^n$ is a \textit{critical point} of $f$ over $\R^n$ if it is a real root of the system of the $n$ partial derivatives of $f$, i.e.:
\[\left\lbrace  \frac{\partial f}{\partial x_1}(x) =0 \,,\cdots,
\,\frac{\partial f}{\partial x_n}(x) = 0 \right\rbrace .\]
The \textit{gradient ideal} of $f$, denoted by $\I_{\nabla}(f)$, 
is generated by all partial derivatives of $f$. From the definition,
the set of critical point of $f$ is the real gradient variety $V_{\R}(\I_{\nabla}(f))$.

We now recall from \cite[Chapter 1]{bertsekas2nd} some fundamental results related to optimality conditions of the unconstrained problem.
Suppose that $x\in \R^n$ is a local minimizer of $f$ over $\R^n$. According to the necessary conditions for optimality \cite[Proposition 1.1.1]{bertsekas2nd}, $x$ belongs to $V_{\R}(\I_{\nabla}(f))$ and $\nabla^2 f$, the Hessian matrix of $f$, is positive semi-definite at $x$. Conversely, if $x$ is a critical point of $f$ and the second-order sufficiency condition holds, i.e. $\nabla^2 f(x)$ is positive definite, then $x$ is a strict local minimizer \cite[Proposition 1.1.2]{bertsekas2nd}.

\subsection{Local minimizers of constrained problems}
Assume that $f,h_1,\dots,h_m$ are polynomials in $\R[x]$, with $1\leq m < n$. The complex algebraic variety defined by $h$ is the following set:
\begin{equation}\label{eq:K}
    V_{\C}(h)=\{x\in\C^n: h_1(x)=\dots=h_m(x)=0\}.
\end{equation}
We recall some notions and fundamental results on local minimizers of $f$ over the real algebraic variety $V_{\R}(h)=V_{\C}(h)\cap\R^n$.

A point $\alpha\in \R^n$ is a \textit{critical point} of $f$ over $V_{\R}(h)$ if there exists  multipliers $\lambda = (\lambda_1,\dots, \lambda_m)\in\R^m$, called a vector of \textit{Lagrange multipliers} associated to $\alpha$, such that
\begin{equation}\label{eq:nablaLagr}
    \nabla f(\alpha) + \lambda_1 \nabla h_1(\alpha)+ \dots + \lambda_m\nabla h_m(\alpha)=0,
\end{equation}
where $\nabla g$ is the gradient vector of the polynomial $g$.
Such a pair $(\alpha,\lambda)$ is a critical point of the \textit{Lagrangian function} that is in $n+m$ variables,
\begin{equation}\label{eq:Lxlambda}
    L(x,\lambda):=f(x)+\lambda^T h(x).
\end{equation}

We denote by $\nabla h=[\nabla h_1 \, \cdots \, \nabla h_m ]^T$ the Jacobian matrice of $h$.  Suppose the matrix $\nabla h$ is full rank at $\alpha$, i.e. $\nabla h_1(\alpha), \dots, \nabla h_m(\alpha)$ are linearly independent, for all $x$ in $V_{\C}(h)$. If $x$ is a local minimizer of $f$ over $V_{\R}(h)$, then
there exists unique $\lambda\in \R^p$ such that $(\alpha, \lambda)$ is a critical point of $L$, and the second-order necessary
condition holds \cite[Proposition 3.1.1]{bertsekas2nd}:
\begin{equation}\label{eq:nablaL0}
    u^T\Big( \nabla^2_{x} L(\alpha, \lambda)\Big)u \geq 0, \ \forall u\in \bigcap_{j=1}^{m}\nabla h_j(\alpha)^{\perp},
\end{equation}
where $\nabla^2_{x} L$ is the Hessian matrix of $L$ with respect to only the variable $x$, $a^{\perp}$ is the orthogonal complement to vector $a\in\R^n$. Conversely, if $(\alpha, \lambda)$ is a critical point of $L$ and the second-order sufficiency condition holds:
\begin{equation}\label{eq:nablaL}
    u^T\Big( \nabla^2_{x} L(x, \lambda)\Big)u > 0, \ \forall 0\neq u\in \bigcap_{j=1}^{m}\nabla h_j(\alpha)^{\perp},
\end{equation}
then $\alpha$ is a strict local minimizer \cite[Proposition 3.2.1]{bertsekas2nd}.
For convenience, the conditions \eqref{eq:nablaL0} and \eqref{eq:nablaL} will be written in the following way, respectively:
\[\nabla^2_{x} L(\alpha, \lambda) \succeq_{\nabla h(\alpha)^{\perp}} 0 \ \text{ and } \ \nabla^2_{x} L(\alpha, \lambda) \succ_{\nabla h(\alpha)^{\perp}} 0.\]

\section{Computing local minimizers of polynomials over $\R^n$}\label{sec:unconstraited}
Let $f$ be a poly\-nomial in $\R[x]$. In this section, we focus on computing local minimizers of $f$ over $\R^n$ under genericity conditions.


\subsection{A technical result}
We provide a technical result which will be used to prove the main result in Theorem \ref{thm:local}.
By adding an extra condition, the below lemma gives us a necessary and sufficient condition for strict local minimizers of $f$.

\begin{lemma}\label{lm:localiff}
    Suppose that $\det(\nabla^2 f)$ has no roots on $V_{\R}(\I_{\nabla}(f))$.
    Then, $\alpha$ in $\R^n$ is a local minimizer of $f$ if and only if $\alpha$ in $V_{\R}(\I_{\nabla}(f))$ and $\nabla^2 f(\alpha)$ is a positive definite matrix.
\end{lemma}

\begin{proof}
Recall that the polynomial $f$ is twice continuously differentiable on $\R^n$. It is easy to see that the sufficient condition is the second-order sufficient condition for local optima\-lity \cite[Proposition  1.1.2]{bertsekas2nd}. We only need to prove the necessary condition.

Let $\alpha\in \R^n$ be a local minimizer of $f$ over $\R^n$. According to \cite[Proposition 1.1.1]{bertsekas2nd}, the second-order necessary condition holds, i.e. $\alpha$ belongs to $V_{\R}(\I_{\nabla}(f))$ and the Hessian matrix $\nabla^2 f(\alpha)$ is positive semi-definite.
From the assumption that $\det(\nabla^2 f)$ has no roots on $V_{\R}(\I_{\nabla}(f))$, the value of $\det(\nabla^2 f)$ at $\alpha$ must be nonzero. It follows that the matrix $\nabla^2 f(\alpha)$ is positive definite.
\end{proof}

\begin{example} We consider the polynomial $f(x_1,x_2)=x_1^2+(x_1x_2-1)^2$ in $\R[x_1,x_2]$. The real gradient variety has only one point, $V_{\R}(\I_{\nabla}(f))=\{(0,0)\}$. The determinant of the Hessian matrix of $f$ is given by $\det(\nabla^2 f)=-12 x_1^2 x_2^2  + 4 x_1^2  + 16 x_1 x_2 - 4$ that is negative at $(0,0)$. According to Lemma \ref{lm:localiff}, $f$ has no local minimizers.
\end{example}

\subsection{Existence of univariate representations}

The following theorem shows that, under genericity conditions, the set of all local minimizers of the polynomial $f$ over $\R^n$ is described by a univariate problem.

\begin{theorem}\label{thm:local} Suppose that $\I_{\nabla}(f)$ is zero-dimensional and $\det(\nabla^2 f)$ has no roots on $V_{\R}\left(\I_{\nabla}(f)\right)$. Then, there exists  an invertible matrix $A$ in $\Mat(n\times n,\R)$, a polynomial $w$ in $\R[t]$, a vector $v$ in $\R[t]^{n}$, and a matrix $H$ in $\Mat(n\times n,\R[t])$ such that
    \begin{equation}\label{eq:local}
        \loc(f,\R^n) = \big\{A^{-1}\big(v(t)\big): w(t)=0, t\in\R,  H(t)\succ 0 \big\}.
    \end{equation}
\end{theorem}

Such $4$-tuple  $\left(A,w,v,H\right)$ is a \textit{univariate representation} of the local minimizers of $f$ over $\R^n$.

\begin{proof}We first prove existence of a univariate representation of the complex critical set of $f$:

{\sc Claim.} \textit{There exist polynomials $w,v_2,\dots,v_n \in \R[t]$ and an invertible matrix $A$ in $\Mat(n\times n,\R)$  such that}
\begin{equation}\label{eq:crit}
    V_{\C}(\I_{\nabla}(f)) = \big\{A^{-1}\big(t,v_2(t),\dots,v_n(t)\big): w(t)=0, t\in\C \big\}.
\end{equation}

We consider the following two cases. First, the points in $V_{\C}(\I_{\nabla}(f))$ have distinct
$x_1$-coordinates. It follows from Lemma \ref{lm:radC} that $\rad(\I_{\nabla}(f))$ is zero-dimensional and radical, and the points in $V_{\C}(\rad(\I_{\nabla}(f)))=V_{\C}(\I_{\nabla}(f))$ have distinct
$x_1$-coordinates. With the lexicographical monomial ordering $x_1< x_2< \cdots < x_n$ in $\R[x]$, according to the Shape Lemma (Lemma \ref{lm:shape0}),  $\rad(\I_{\nabla}(f))$ is in shape position, i.e. there exist $w, v_2,\dots, v_n$ in $\R[x_1]$ of degrees $\deg v_i<\deg w=\# V_{\C}(\I_{\nabla}(f))$,  for $i=2,\dots,n$, such that
\[\rad(\I_{\nabla}(f))=\left\langle
w(x_1),x_2-v_2(x_1),
\dots,x_n-v_n(x_1)\right\rangle.\]
Thus, $V_{\C}(\I_{\nabla}(f))$ can be parameterized as follows:
\begin{equation}\label{eq:shape}
    V_{\C}(\I_{\nabla}(f)) = \big\{\big(t,v_2(t),\dots,v_n(t)\big): w(t)=0 \big\}.
\end{equation}
In this case, one defines $A=I_n$ -- the identical matrix. Hence, \eqref{eq:shape} becomes \eqref{eq:crit}.

Second, there are two points in $V_{\C}(\I_{\nabla}(f))$ having the same
$x_1$-coordinates, we handle this case by changing variables as in the proof of \cite[Theorem 3.1]{magron2023} as follows:

Denoting by $\delta$ the cardinality of $V_{\C}(\I_{\nabla}(f))$, according to \cite[Lemma 2.1]{rouillier1999},
there exists a number
$j$ belonging the set $\{1,\dots,(n-1)\delta(\delta-1)/2\}$ such that the linear function
$u(x)=x_1+jx_2+\dots+j^{n-1}x_n$ separates $V_{\C}(\I_{\nabla}(f))$, i.e., $u(\alpha)\neq
u(\beta)$ for any distinct points $\alpha,\beta$ in $V_{\C}(\I_{\nabla}(f))$. We consider the
change of variables $y=Ax$, where
\begin{equation}\label{T}
    A=\begin{bmatrix}
        1 & j & j^2 &\cdots & j^{n-1} \\
        0 & 1 & 0 & \cdots  & 0 \\
        0& 0 & 1 & \cdots & 0\\
        \vdots & \vdots & \vdots & \ddots & \vdots\\
        0& 0 & 0 & \cdots & 1
    \end{bmatrix}.
\end{equation}
Clearly, $A$ is an invertible matrix. Then, we obtain a new polynomial $g$ in variables $y_1, y_2, \dots, y_n$, says
$g(y)=f(A^{-1}y)$. Thanks to the chain rule $\nabla g= \nabla f \circ
A^{-1}$, $V_{\C}(\I_{\nabla}(g))$ is the image of $V_{\C}(\I_{\nabla}(f))$ under the linear map associated to matrix $A$.
Conversely, one has
\[V_{\C}(\I_{\nabla}(f)) = \Big\{\alpha\in\C^n: \alpha=A^{-1}\beta, \, \beta\in V_{\C}(\I_{\nabla}(g))\Big\}.\]
Hence, this set is finite. Since
$y_1=u(x)$ separates
$V_{\C}(\I_{\nabla}(g))$, the
distinct points in $V_{\C}(\I_{\nabla}(g))$ have distinct
$y_1$-coordinates. By applying the conclusion in the first case for $g$, we obtain the conclusion with matrix $A$ given in \eqref{T}. The claim is proved.

Since $\det(\nabla^2 f)$ has no roots on $V_{\R}\left(\I_{\nabla}(f)\right)$, according to Lemma \ref{lm:localiff}, the following equality holds
\begin{equation}\label{eq:localf}
    \loc(f,\R^n) = \big\{\alpha\in V_{\R}\left(\I_{\nabla}(f)\right): \nabla^2 f (\alpha)\succ 0 \big\}.
\end{equation}

\if
Thanks to the chain rule $\nabla g= \nabla f \circ
A^{-1}$, one has $\nabla^2 g= \nabla^2 f \circ
A^{-1}$. Thus
\[\det(\nabla^2 g)= \det(\nabla^2 f) \times \det(A^{-1}) = \det(\nabla^2 f)\]
since $\det(A)=1$.
\fi

We denote by \begin{equation}\label{eq:vt1}
    v(t):=\big(t,v_2(t),\dots,v_n(t)\big).
\end{equation}
Thanks to above claim, one has
\begin{equation}\label{eq:Vf}
    V_{\R}(\I_{\nabla}(f)) = \big\{A^{-1}\big(v(t)\big): w(t)=0, t\in\R \big\}.
\end{equation}
Matrix $H(t)$ is obtained by replacing $y$ by $v(t)$ in $\nabla^2 g$, says $ H(t) = \nabla^2 g  \big(v(t)\big)$. As the change of variables, we have
\begin{equation}\label{eq:Ht}
    H(t) = \nabla^2 f (A^{-1}\big(v(t)\big)).
\end{equation}

Combining the results in \eqref{eq:localf}, \eqref{eq:Vf}, and \eqref{eq:Ht}, we obtain \eqref{eq:local}.
\end{proof}

\begin{remark} We discuss genericity of the two assumptions in Theorem \ref{thm:local}. Recall that a property is said to be \textit{generic} on $\R^m$ if it holds on a subset $\Oo \subset \R^m$ such that its complement $\R^m\setminus\Oo$ has Lebesgue measure zero.
    As mentioned in \cite[Proposition 1]{nie2006}, the first assumption of Theorem \ref{thm:local} related to being zero-dimensional of $\I_{\nabla}(f)$ is generic in $\R[x]_d$, the linear space of all polynomials in $\R[x]$ of degree at most $d$. Moreover, the property of being non-zero on the critical set $V_{\R}\left(\I_{\nabla}(f)\right)$ of $\det(\nabla^2 f)$ is also generic in the linear space
    \[\mathcal{A}_\mathcal{N}=\{f\in\R[x]: \mathcal{N}(f) \subseteq \mathcal{N}\},\]
    where $\mathcal{N}$ is a given Newton polyhedron at infinity and $\mathcal{N}(f)$ is the Newton polyhedron at infinity of $f$ \cite[Theorem 5.1]{pham2016genericity}.

\end{remark}
\smallskip

We illustrate Theorem \ref{thm:local} by considering the following simple example.

\begin{example}\label{ex:2} We consider the polynomial $f$ in two variables with real coefficients $ f(x_1,x_2)= x_1^2+x_2^4-2x_2^2$.
    The gradient ideal $\I_{\nabla}(f)$ is
    \begin{equation}\label{Igradf}
        \I_{\nabla}(f)=\left\langle 2x_1, 4x_2^3 - 4x_2
        \right\rangle
    \end{equation}
    that is zero-dimensional but not in shape position w.r.t the lexicographical monomial ordering $x_1<x_2$. There are points in \begin{center}
        $V_{\C}(\I_{\nabla}(f))=\left\lbrace (0,-1),(0,0),(0,1)\right\rbrace$
    \end{center} have the same $x_1$-coordinates. We choose $j=1$, then the linear function
    $u=x_1+x_2$ separates $V_{\C}(\I_{\nabla}(f))$. Consider the
    change of variables $y=Ax$, where
    \begin{equation*}\label{AA}
        A=\begin{bmatrix}
            1 & 1 \\
            0 & 1
        \end{bmatrix}, \text{ and }
        A^{-1}=\begin{bmatrix}
            1 & -1 \\
            0 & 1
        \end{bmatrix},
    \end{equation*}
    we obtain the polynomial $g$ in variables $y_1, y_2$,
    \[g(y)=f(A^{-1}y)=y_1^2+y_2^4-y_2^2-2y_1y_2.\]
    By an easy computation, we obtain $\I_{\nabla}(g) = \left\langle y_1^3-y_1,y_2-y_1 \right\rangle $, the ideal is in shape position w.r.t. the lexicographical monomial ordering $y_1<y_2$.  One has $w(t)= t^3-t$ and $v_{2}(t)=t$, thus $A^{-1}(v(t))=(0,t)$, where $v(t)=(t,v_2(t))$. The set of real roots of $w$ is $\{-1,0,1\}$.

    The Hessian matrix of $f$ with respect to variables $(x_1,x_2)$ is given by
    \[\nabla^2 f=
    \begin{bmatrix}
        2  & 0 \\
        0 & 12x_2^2-4
    \end{bmatrix}, \ \text{ hence } \ H(t)= \nabla^2 f (0,t)= \begin{bmatrix}
        2  & 0 \\
        0 & 12t^2-4
    \end{bmatrix}.\]
    The determinant of  $\nabla^2 f$ is $8(3x_2^2-1)$.
    This determinant is nonzero on the real gradient variety $V_{\R}(\I_{\nabla}(f))$. Moreover, $H(t)$ is positive definite at two points $t=-1$ and $t=1$. Thanks to \eqref{eq:local}, we conclude that the set of all local minimizes of $f$ over $\R^2$ is
    \[\loc(f) = \big\{\big(0,t\big): w(t)=0, t\in\R,  H(t)\succ 0 \big\}=\{(0,-1),(0,1)\}.\]
\end{example}

If we are only interested in global minimizers and the minimum value, the
assumptions in Theorem \ref{thm:local} can be lightened. In particular, the assumption $\det(\nabla^2 f)$ being non-zero on $V_{\R}\left(\I_{\nabla}(f)\right)$ can be dropped. If $\I_{\nabla}(f)$ is zero-dimensional and the infimum of $f$ over $\R^n$, denoted by $f_{\min}$, is attained, we show in the following theorem that this value can be computed though a constrained univariate optimization problem.

\begin{theorem}\label{thm:global}
    Suppose that $\I_{\nabla}(f)$ is zero-dimensional and the infimum of $f$ over $\R^n$ is attained. Then, there exist two univariate polynomials $r, w\in\R[t]$ of degrees $\deg r < \deg w = \#V_{\C}(\I_{\nabla}(f))$ such that
    \begin{equation}\label{eq:rmin1}
        f_{\min}=\min \{r(t):w(t)=0, t\in\R\}.
    \end{equation}
\end{theorem}
\begin{proof}
Since $\I_{\nabla}(f)$ is zero-dimensional, according to the claim in the proof of Theorem \ref{thm:local}, there exist $w,v_2,\dots,v_n \in \R[t]$, denote by $v(t)=(t,v_2,\dots,v_n)$, and an invertible matrix $A$ in $\Mat(n\times n,\R)$  such that \eqref{eq:crit} holds, then \eqref{eq:Vf} so is. We define
\begin{equation}\label{eq:pt}
    p(t):=f(A^{-1}\big(v(t)\big)).
\end{equation}
Clearly, $p$ is a polynomial in variable $t$ with real coefficients.
Let $r$ be the remainder of the division of $p$ by $w$ with $\deg r < \deg w$.

We now demonstrate the following relationship:
\begin{equation}\label{eq:minmin}
    f_{\min}=\min \{p(t):w(t)=0, t\in\R\}=\min \{r(t):w(t)=0,t\in\R\}.
\end{equation}
The first equality of \eqref{eq:minmin}
is obtained from the following relations:
$$\begin{aligned}
    \min_{x} \{f(x):x\in \R^n\} & = \min_{x} \left\lbrace f(x):x\in V_{\R}(\I_{\nabla}(f))\right\rbrace  \\
    &= \min_x \left\lbrace f(x):x\in V_{\R}(\rad(\I_{\nabla}(f)))\right\rbrace \\
    &=\min_{t} \left\lbrace f(A^{-1}\big((v(t)\big)):w(t)=0,t\in\R\right\rbrace \\
    &= \min_t \left\lbrace p(t):w(t)=0,t\in\R\right\rbrace .
\end{aligned}$$
We give here some explanations for above equalities. The first equality holds because every global minimizer is in the critical set, thanks to the assumption. The second  equality is obtained relying on Lemma \ref{lm:radC}. The third one is implied from \eqref{eq:Vf}. The last equality is obtained from the definition of $p$.

The last equality in \eqref{eq:minmin} is implied from the fact that $p=gw+r$ and that the values of $p$ and $r$ are coincident on the set $\{t\in\R:w(t)=0\}$.
\end{proof}

\begin{remark}\label{rmk:degree1}  It is worth emphasizing here the meaning of the equality in \eqref{eq:rmin1}. Suppose that $f$ has finitely many complex critical points and $f$ reaches its infimum. From Bezout's theorem
    \cite[Chapter 8, \S 7]{cox2013}, $\#V_{\C}(\I_{\nabla}(f))$
    is at most $(d-1)^n$, where $d$ is the degree of $f$. Since  $\deg v_i<\deg w$ for $i=2,\dots, n$, the degree of $p$ given in \eqref{eq:pt} does not exceed $d(d-1)^{n}$, this number is significantly larger than $(d-1)^n-1$ that is the possibly highest degree of $r$. It turns out from \eqref{eq:rmin1} that
    computing the minimum of the polynomial $f$ in $n$ variables of degree $d$ over $\R^n$ is equivalent to computing the minimum of a univariate polynomial $r$ of degree at most $(d-1)^n-1$ over the real roots of another univariate polynomial $w$, with $\deg r < \deg w \leq (d-1)^n$. Thanks to the last equality in \eqref{eq:minmin}, the complexity of computing $\min \{r(t):w(t)=0\}$ is less than $d$ times compared to $\min \{p(t):w(t)=0\}$.
\end{remark}

\begin{example}\label{ex:3} We consider polynomial $f$ in Example \ref{ex:3}. This polynomial is coercive, thus its infimum value is reached. From the results in Example \ref{ex:3}, one has $w(t)= t^3-t$, $v_{2}(t)=t$, and $A^{-1}(t,v_2(t))=(0,t)$. Thus the polynomial $p$ and remainder $r$ of the division of $p$ by $w$ are given respectively by $p(t)=f(0,t)=t^4-2t^2$ and $r(t)=-t^2$.
    According to the formula \eqref{eq:rmin1} in Theorem \ref{thm:global}, one has
    \[f_{\min}=\min \{-t^2:t^3-t=0\} =-1,\]
    hence the set of global minimizers is $\{(0,-1),(0,1)\}$ that is coincident with the set of local minimizers $\loc(f,\R^2)$.
\end{example}

\subsection{Algorithms computing local minimizers and the minimum}\label{sub:compRn}
Based on the proof of Theorem \ref{thm:local}, we design an algorithm named $\grulom$ to compute all LOcal Minimizers of a polynomial $f$ over $\R^n$ (Unconstrained) though its GRadient ideal, and another one named $\grulomplus$, based on the proof of Theorem \ref{thm:global}, to compute the minimum if it exists.

\subsubsection{Algorithm $\grulom$}
The input of Algorithm $\grulom$ is a polynomial $f$ in $\R[x]$ such that the assumptions in Theorem \ref{thm:local} are satisfied.
The output of  $\grulom$ is the set of local minimizers of $f$ over $\R^n$.

\textit{Description}. At the beginning, we set $j=0$ and $\shape=$ false, the variable $\shape$ is boolean.
The while loop at Line~2 will stop if $\shape$ is true. Within the while loop, at Line 3, with $j$ given, we define matrix $A$  depending on $j$ as in \eqref{eq:A} and change the variables and obtain $g(y)= f(A^{-1}y)$.

At Line~4, we compute the radical of the ideal $\I_{\nabla}(f)$ and its reduced Gr\"{o}bner basis $G$ in $\R[y_1,\dots,y_n]$ with respect to the ordering with $y_1< y_2< \cdots < y_n$. At Line 5, we set the value of $\shape$ as true if
$\rad(\I_{\nabla}(f))$ is in shape position
otherwise, we increase the value of $j$ by setting $j:=j+1$.

By substituting $y=v(t)$ in $\nabla^2 g(y)$, we obtain the univariate polynomial matrix $H(t)$ at Line 6.
The notation $H(t):=\tilde{H}(t)\ \text{mod} \ w(t)$ means $H_{ij}(t):=\tilde{H}_{ij}(t)\ \text{mod} \ w$ for all $i,j=1,\dots,n$.

\begin{algorithm}
    \caption{$\grulom$ computing local minimizers of a polynomial over $\R^n$}\label{alg:local}

    \smallskip

    \textbf{Input:} $f\in \R[x]$ such that  $\I_{\nabla}(f)$ is zero-dimensional and $\det(\nabla^2 f)$ is non-zero on $V_{\R}\left(\I_{\nabla}(f)\right)$
    \smallskip

    \textbf{Output:} $\loc(f,\R^n)$ -- the set of all local minimizers of $f$ over $\R^n$

    \smallskip
    \begin{enumerate}
        \item [\rm 1:] Set $j:=0$ and $\shape:=$ false
        \smallskip
        \item [\rm 2:] While $\shape=$ false do
        \smallskip
        \begin{itemize}
            \item [\rm 3:] Change the variables $x=A^{-1}y$ in $f(x)$ and obtain $g(y)$, where
            \begin{equation}\label{eq:A}
                A:=\begin{bmatrix}
                    1 & j & j^2 &\cdots & j^{n-1} \\
                    0 & 1 & 0 & \cdots  & 0 \\
                    0& 0 & 1 & \cdots & 0\\
                    \vdots & \vdots & \vdots & \ddots & \vdots\\
                    0& 0 & 0 & \cdots & 1
                \end{bmatrix}
            \end{equation}
            \item [\rm 4:] Compute $\rad(\I_{\nabla}(g))$ in $\R[y_1,\dots,y_n]$ and its reduced Gr\"{o}bner basis $G$  with respect to the lexicographical monomial ordering $y_1<  \cdots < y_n$
            \smallskip
            \item [\rm 5:] If the basis $G$ has the following form \begin{equation}\label{eq:G}
                G=[w,y_2-v_2,\dots,y_n-v_n],
            \end{equation}
            where $w,v_2,\dots,v_n$ are in $\R[y_1]$, i.e. $\rad(\I_{\nabla}(g))$
            is in shape position, then set $\shape:=$ true, else
            set $j:= j+1$. Denote $v(t):=(t,v_2(t),\dots,v_n(t))$.
            \smallskip
        \end{itemize}

        \item [\rm 6:] Compute $\nabla^2 g(y)$ and obtain $\tilde{H}(t)$ by substituting $y=v(t)$ in $\nabla^2 g(y)$, where $v_i$'s are from Line 5, and then compute $H(t):=\tilde{H}(t) \ \text{mod} \ w(t)$
        \smallskip

        \item [\rm 7:] Return the set
\[\loc(f,\R^n) = \Big\{A^{-1}\big(v(t)\big): w(t)=0, t\in\R, H(t)\succ 0 \Big\}\]
    \end{enumerate}
\end{algorithm}

The correctness of Algorithm $\grulom$ is stated and proven in the following proposition.

\begin{proposition}\label{thm:corectness} Let $f$ be in $\R[x]$ such that the conditions in Theorem~\ref{thm:local} hold. On input $f$, Algorithm $\grulom$
    terminates and returns all local minimizers of $f$.
\end{proposition}

\begin{proof}
Thanks to \cite[Lemma 2.1]{rouillier1999}, with the matrix $A$ defined by \eqref{eq:A}, the while loop will be certainly ended within $(n-1)\delta(\delta-1)/2$ iterations, where $\delta$ is the cardinality of $V_{\C}(\I_{\nabla}(f))$.

When the while loop ends, from the assumptions, the ideal $\rad(\I_{\nabla}(g))$ is
zero-dimensional and radical. Its reduced Gr\"{o}bner basis $G$ has the form \eqref{eq:G}, then $\rad(\I_{\nabla}(f))$ is in shape position
with \[\rad(\I_{\nabla}(f))=\left\langle w,y_2-v_2,\dots,y_n-v_n\right\rangle.\]

By the definition, the values of $H_{ij}(t)$  and $\tilde{H}_{ij}(t)$ are the same on the roots of $w(t)$. It follows that $H(t)\succ 0$ is equivalent to $\tilde{H}(t)\succ 0$ for $t$ satisfying $w(t)=0$.

According to Theorem \ref{thm:local}, the set of all local minimizers of $f$ over $\R^n$ is the following set:  $\{A^{-1}(t,v_2(t),\dots,v_n(t)): w(t)=0, t\in\R, H(t)\succ 0\}$.
\end{proof}

\begin{remark} In \cite{safey2008computing,greuet2014}, the authors design \textit{exact} algorithms computing a univ\-ariate polynomial
    vanishing at the global infimum and an isolating interval for the value. Our algorithm $\grulom$ provides a \textit{univariate representation} $\left(A,w,v, H\right)$ for the local minimizers of $f$ over $\R^n$ such that the set of local minimizers of $f$ is described as in \eqref{eq:local}.
    Furthermore, by the definition at Line 6, one has $\deg H_{ij}(t) <\deg w$. Hence, we claim that the degrees of all polynomials $v_i,H_{ij}$ do not exceed $(d-1)^n-1$.
\end{remark}

\begin{remark}\label{rmk:maple}
    We would like to remark that a computer algebra system such as {\sc Maple} with the package $\mathtt{Groebner}$ can handle symbolic tasks in  Algorithm $\grulom$. To compute the radical of an ideal, we use the command $\mathtt{Radical}$; to compute the reduced Gr\"{o}bner basis of an ideal with respect to a given lexicographical order, we use the command $\mathtt{Basis}$. From the point of view of complexity, this is the most
    computationally expensive part of Algorithm $\grulom$.

    The input conditions of Algorithm $\grulom$ are verifiable in {\sc Maple}. To be more specific,  to check whether an ideal $\I$ is zero-dimensional or not, we use the command $\mathtt{IsZeroDimensional}$.
    One sees that $\det(\nabla^2 f)$ is non-zero on $V_{\R}\left(\I_{\nabla}(f)\right)$ if and only if the following polynomials
    \[\left\lbrace \frac{\partial f}{\partial
        x_1},\cdots, \frac{\partial
        f}{\partial x_n}, \det(\nabla^2 f)\right\rbrace\]
    have no real roots. To verify this property with rational coefficients, we use the command $\mathtt{HasRealRoots}$ within the package $\mathtt{RootFinding}$.
\end{remark}

\subsubsection{Algorithm $\grulomplus$}
The input of this algorithm is a polyno\-mial $f$ in $\R[x]$ such that $\I_{\nabla}(f)$ is zero-dimensional and $f$ attains its infimum which is denoted by $f_{\min}$.
The output of  $\grulomplus$ is the set of global minimizers $\glo(f)$ and the value $f_{\min}$.

\textit{Description}.  At Line~1, we perform the tasks at Lines 1--5 in Algorithm $\grulom$ to obtain matrix $A$, polynomial $w$, and vector of polynomials $v$. At Line~2, we compute the univariate polynomials $p$ and $r$, where $r$ is the remainder of the division of $p$ by $w$. At Line~3,
we compute the minimum of $r$ with the constraint $\{t\in\R:w(t)=0\}$ and the set $\glo(f,\R^n)$.

\begin{algorithm}
    \caption{$\grulomplus$ computing global minimizers of a polynomial over $\R^n$}\label{alg:global}
    \smallskip

    \textbf{Input:} $f\in \R[x]$ such that
    $\I_{\nabla}(f)$ is zero-dimensional and $f_{\min}$ attains
    \smallskip

    \textbf{Output:} $\glo(f,\R^n)$ the set of global minimizers and $f_{\min}$ the minimum value

    \begin{itemize}
        \item [\rm 1:] Compute as at Lines 1--5 in Algorithm $\grulom$ to obtain $A$, $w$, and $v$
        \smallskip

        \item [\rm 2:] Set $p(t):=f(A^{-1}\big(v(t)\big))$ and compute the remainder $r$ of the division of $p$ by $w$
        \smallskip

        \item [\rm 3:] Compute $f_{\min}=\min \{r(t):w(t)=0, t\in\R\}$ and
        \[\glo(f,\R^n):=\left\lbrace A^{-1}\big(v(t)\big): r(t)=f_{\min}, w(t)=0, t\in\R\right\rbrace \]

        \smallskip

        \item [\rm 4:] Return $\glo(f,\R^n)$ and $f_{\min}$
    \end{itemize}
\end{algorithm}

\begin{proposition}\label{prop:corectness2} Let $f$ be in $\R[x]$ such that such that $\I_{\nabla}(f)$ is zero-dimensional and $f_{\min}$ attains. On input $f$, Algorithm $\grulomplus$
    terminates and returns the set of global minimizers and $f_{\min}$.
\end{proposition}

\begin{proof}
The tasks at Lines 1--5 in Algorithm $\grulom$ work under the assumption that $\I_{\nabla}(f)$ is zero-dimensional. We obtain $A$, $w$, and $v$ such that $V_{\C}\left(\I_{\nabla}(f)\right)$ is given as in \eqref{eq:crit}.

Since $r$ is the remainder of the division of $p$ by $w$, the values of $r$ and $p$ are coincident on $\{t\in\R:w(t)=0\}$.
According to Theorem \ref{thm:global}, the minimum of $f$ over $\R^n$ is $\min \{r(t):w(t)=0\}$. Clearly, $\glo(f,\R^n)$ defined at Line 3 is the set of all global minimizers of $f$.
\end{proof}

\subsection{Experimental examples}
The algorithms are implemented in {\sc Maple}, and the results are obtained on an Intel i7 - 8665U CPU (1.90GHz) with 32 GB of RAM.

Within our implementation, we use \textit{hybrid symbolic-numerical computations}. The tasks at Lines 1--3 in $\grulom$ are done symbolically.
The tasks at Line 4 are performed numerically by isolating real roots of $w(t)$. We set the variable $\texttt{Digits}:=\deg w + 10$, where $\texttt{Digits}$ controls the number of digits that {\sc Maple} uses when making calculations with software floating-point numbers.

We perform here our experiment with the family of Rosenbrock functions that are usually used as a performance test problem for optimization algorithms, for instance \cite{conn1988testing,wang2022cs,shang2006note}. These functions are polynomial, sparse, and non-convex:
\[f_n=\sum_{i=1}^{n-1}\left[ 100(x_i^2-x_{i+1})^2 + (x_i^2-1)^2 \right], \ n\geq 2.\]
Clearly, $f_n$ is in $n$ variables, of degree $4$, and non-negative over $\R^n$. Moreover, the coefficients are rational. Its minimum is $0$ and reached at the point $\alpha^A=(1,\dots,1)$ in $\R^n$.

Let $n$ runs from $2$ to $7$. By checking symbolically, the conditions in Theorem \ref{thm:local} are satisfied for these functions. We obtain the following table: \\

\begin{center}
    {\small \begin{tabular}{|c|c|c|c|c|c|c|c|c|}
            \hline
            $n$ & $\deg w$ & $\deg h$ & $\deg r$ &  \#real$(w)$ & $\#\loc$ & $\alpha^A_{1}$ & $\alpha^B_{1}$ & time (s) \\ \hline
            2   &    1     &     4    &  0    & 1 & 1 & 1.0   & -- &  0.12     \\ \hline
            3   &    3     &     8    &  2    & 1 & 1 & 1.0   & -- &  0.14    \\ \hline
            4   &    9     &     32   &  8    & 3 & 2 & 1.0  & -0.77565 &  0.28    \\ \hline
            5   &    27    &    104   &  26   & 3 & 2 & 1.0  & -0.96205 &  0.91    \\ \hline
            6   &   81     &   320    &  80   & 3 & 2 & 1.0  & -0.98657 &  2.07    \\ \hline
            7   &   243    &    968   &  242  & 3 & 2 & 1.0  & -0.99172 &  73.11   \\ \hline
        \end{tabular}
    } \\
    \bigskip

    {\small \textbf{Table 1.} $\grulom$ computing local minimizers of Rosenbrock functions over~$\R^n$}
\end{center}

\medskip

In Table 1, we have information of the degrees of univariate polynomials $w$, $h$, and $r$; \#real$(w)$ is the number of real roots of $w$; $\#\loc$ is the number of local minimizers, the time is in second (s). Denote $\alpha^A_{1}$ and $\alpha^B_{1}$ by the first coordinates of the local minimizers $\alpha^A$ and $\alpha^B$, respectively; Compared to the results in \cite[Table 1]{shang2006note}, they are the same.

The computation time grows quickly when we increase $n$ the number of variables. Symbolic task such as computing the reduced Gr\"{o}bner basis of an ideal is the most computationally expensive part of Algorithm $\grulom$.


\section{Computing local minimizers with equality constraints}\label{sec:eq}
Let $f$ and $h_1,\dots,h_m$ be polynomials in $\R[x]$. The present section considers the minimization problem whose objective function is $f$ and constraint set is the real algebraic set $V_{\R}(h)$. Throughout this paper, we assume that the last set is nonempty.
The techniques used in this section are almost the same techniques in the previous section, hence several proofs will be omitted or sketched out.

\subsection{A technical result}

Similarly to the unconstrained case, we need to prove a technical result which is essential to prove the main results.

Recall that $L$ is the Lagrangian function defined in \eqref{eq:Lxlambda}; $\I_{\nabla}(L)$ is the gradient ideal of $L$; $\nabla^2_{x} L$ is the Hessian matrix of $L$ with respect to only the variable $x$, hence its determinant is a polynomial in variables $(x,\lambda)$.

\begin{lemma}\label{lm:localiff2}
Suppose that $\nabla h$ is full rank on $V_{\R}(h)$
and $\det(\nabla^2_{x} L)$ has no roots on $V_{\R}(\I_{\nabla}(L))$.
    Then, $\alpha$ in $V_{\R}(h)$ is a local minimizer of $f$ over $V_{\R}(h)$ if and only if there is a Lagrange multiplier vector $\lambda$ such that $\nabla^2_{x} L(\alpha, \lambda) \succ_{\nabla h(\alpha)^{\perp}} 0$.
\end{lemma}

\begin{proof}
Clearly, the sufficient condition is the second-order sufficient condition for optima\-lity \cite[Proposition 3.2.1]{bertsekas2nd} with a note that polynomials $f,h_1,\dots,h_m$ are twice continuously differenti\-able on $\R^n$. Thus, we only need to prove the necessary condition.

Suppose that $\alpha$ in $V_{\R}(h)$ is a local minimizer of $f$ over $V_{\R}(h)$. Since $\nabla h(\alpha)$ is full rank, according to the second-order necessary condition for optimality \cite[Proposition 3.1.1]{bertsekas2nd},
there exists 
Lagrange multipliers $\lambda\in \R^p$ such that
\[\nabla f(\alpha) + \lambda_1 \nabla h_1(\alpha)+ \dots + \lambda_m\nabla h_m(\alpha)=0,\]
and $\nabla^2_{x} L(\alpha, \lambda) \succeq_{\nabla h(\alpha)^{\perp}} 0$ holds.
Thanks to the assumption that $\det(\nabla^2_{x} L)$ has no roots on $V_{\R}(\I_{\nabla}(L))$, we conclude that the value of $\det(\nabla^2_{x} L)$ at $(\alpha,\lambda)$ is nonzero.

Suppose on the contrary that $\nabla^2_{x} L(\alpha, \lambda) \succ_{\nabla h(\alpha)^{\perp}} 0$ does not hold. From the fact that $\nabla^2_{x} L(\alpha, \lambda) \succeq_{\nabla h(\alpha)^{\perp}} 0$, there exists $0 \neq v \in \nabla h(\alpha)^{\perp}$ such that $v^T\big( \nabla^2_{x} L(\alpha)\big)v = 0$. It follows that the rank of the $n\times n$--matrix $\nabla^2_{x} L(\alpha,\lambda)$ is strictly smaller than $n$. This contradicts $\det(\nabla^2_{x} L(\alpha,\lambda))\neq 0$.
\end{proof}

\subsection{Existence of univariate representations}
We prove existence of a univ\-ariate representation of local minimizers of $f$ over the real algebraic set $V_{\R}(h)$ under certain conditions. Moreover, if its infimum attains then this value can be computed though a univariate optimization problem.

Denote by $\pi_n: \R^n\times\R^p \to \R^n$ the projection onto the first $n$ coordinates, i.e. $\pi_n(a,b)=a$ for $(a,b)$ in $\R^n\times\R^p$.

\begin{theorem}\label{thm:local2} Suppose that $\nabla h$ is full rank on $V_{\R}(h)$, $\I_{\nabla}(L)$ is zero-dimen\-sional, and $\det(\nabla^2_{x} L)$ has no roots on $V_{\R}\left(\I_{\nabla}(L)\right)$. Then, there exists an invertible matrix $A$ in $\Mat((n+m)\times (n+m),\R)$, a polyno\-mial $w$ in $\R[t]$, a vector $v$ in $\R[t]^{n+m}$, and a matrix $H$ in $\Mat(n\times n,\R[t])$ such that
    \begin{equation}\label{eq:local2}
        \loc(f,V_{\R}(h)) =  \left\lbrace \pi_n\left(A^{-1}(v(t))\right): w(t)=0,t\in\R, H(t)\succ_{\nabla h(t)^{\perp}} 0\right\rbrace .
    \end{equation}
\end{theorem}

\begin{proof}Since $\I_{\nabla}(L)$ is zero-dimensional, by repeating the argument to prove Claim in the proof of Theorem \ref{thm:local}, we also conclude that there exist $n+m$ polynomials $w,v_2,\dots,v_{n+m}$ in $\R[t]$ and an invertible matrix $A$ in $\Mat((n+m)\times (n+m),\R)$  such that
\begin{equation}\label{eq:crit2}
V_{\C}(\I_{\nabla}(L)) = \big\{A^{-1}\big(v(t)\big): w(t)=0, t\in\C \big\},
\end{equation}
where $v(t)=\big(t,v_2(t),\dots,v_{n+m}(t)\big)$. To obtain $H(t)$, we replace $(x,\lambda)$ by $A^{-1}\big(v(t)\big)$ in $\nabla^2_{x} L(x, \lambda)$.

Thanks to Lemma \ref{lm:localiff2} and above, $\alpha\in\R^n$ is a local minimizer if any only if it is the image of a vector $A^{-1}(v(t))$, satisfying $w(t)=0$, $t\in\R$, and $H(t)\succ_{\nabla h(\alpha)^{\perp}} 0$, under the projection $\pi$. Therefore, we obtain \eqref{eq:crit2}.
\end{proof}

Similar to the unconstrained case, if the infimum of $f$ over $V_{\R}(h)$ is attained, then this value can be computed though a univariate problem.

\begin{theorem}\label{thm:global2}
    Suppose that $\nabla h$ is full rank on $V_{\R}(h)$ and $\I_{\nabla}(L)$ is zero-dimen\-sional. If the infimum of $f$ over $V_{\R}(h)$ is attained, then there exist two univariate polynomials $r, w\in\R[t]$ of degrees $\deg r < \deg w = \#V_{\C}(\I_{\nabla}(f,h))$ such that
    \begin{equation}\label{eq:rmin2}
        f_{\min}=\min \{r(t):w(t)=0, t\in\R\}.
    \end{equation}
\end{theorem}
\begin{proof} Since $\I_{\nabla}(L)$ is zero-dimen\-sional, from the proof of Theorem \ref{thm:global2}, there exists a matrix $A$ in $\Mat((n+m)\times (n+m),\R)$, a polynomial $w$ in $\R[t]$, and a vector $v(t)$ in $\R[t]^{n+m}$ such that
\[V_{\R}(\I_{\nabla}(L)) = \big\{A^{-1}\big(v(t)\big): w(t)=0, t\in\R \big\}.
\]
Moreover, as $\nabla h$ is full rank on $V_{\R}(h)$, if $\alpha$ in $V_{\R}(h)$ is a global minimizer then $(\alpha,\lambda)$ is a real root of $\nabla L$ for some $\lambda$. Hence, $f_{\min}=\min \{p(t):w(t)=0, t\in\R\}$, where $p(t) =f\big(\pi_n(A^{-1}v(t))\big)$. The polynomial $r$ is the reminder of the division of $p$ by $w$.
\end{proof}

\begin{remark}
    In \cite[Proposition 2.1]{nie2009algebraic}, Nie and Ranestad pointed out that, for generic polynomials $f$ and $h$, the Jacobian matrix  $\nabla h$ is full rank on $V_{\C}(h)$ and the ideal $\I_{\nabla}(L)$ is zero-dimensional.
\end{remark}

\subsection{Algorithms to compute local and global minimizers} We design an
algorithm named $\gralom$ to compute all LOcal Minimizers of the polynomial $f$ over the real Algebraic set $V_{\R}(h)$ though their GRadient ideals based on Theorem \ref{thm:local2}.

\subsubsection{Algorithm $\gralom$}
The input of $\gralom$ includes polynomials
$f$ and $h_1,\dots,h_m$ in $\R[x]$ such that all assumptions in Theorem \ref{thm:local2} are satisfied.
The output of $\gralom$ includes the set of all local minimizers of $f$ over $V_{\R}(h)$.

All lines in the algorithm are almost the same in $\grulom$ with noticing that the polynomial $L$ has $n+m$ variables $(x,\lambda)$. Hence $\nabla^2_{x} L$, at Line 6, might contain both $x$ and  $\lambda$.

\begin{algorithm}
    \caption{$\gralom$ computing local minimizers of $f$ over $V_{\R}(h)$}\label{alg:local2}

    \smallskip

    \textbf{Input:} $f,h_1,\dots,h_m\in \R[x]$ such that $\nabla h$ is full rank over $V_{\R}(h)$, $\I_{\nabla}(L)$ is zero-dimensional, and $\det(\nabla^2_{x} L)$ is nonzero on $V_{\R}\left(\I_{\nabla}(L)\right)$

    \smallskip

    \textbf{Output:} $\loc(f,V_{\R}(h))$ -- the set of all local minimizers of $f$ over $V_{\R}(h)$

    \smallskip

    \begin{itemize}

        \item [\rm 1:] Set $j:=0$ and $\shape:=$ false
        \smallskip
        \item [\rm 2:] While $\shape=$ false do
        \smallskip
        \begin{itemize}
            \item [\rm 3:] Change the variables $(x,\lambda)=A^{-1}(y)$ in $L(x,\lambda)$ and obtain $\ell(y)$, where
\[A:=\begin{bmatrix}
                    1 & j & j^2 &\cdots & j^{n+m-1} \\
                    0 & 1 & 0 & \cdots  & 0 \\
                    0& 0 & 1 & \cdots & 0\\
                    \vdots & \vdots & \vdots & \ddots & \vdots\\
                    0& 0 & 0 & \cdots & 1
                \end{bmatrix}
\]
            \item [\rm 4:] Compute $\rad(\I_{\nabla}(\ell))$ in $\R[y_1,\dots,y_{n+m}]$ and compute its reduced Gr\"{o}bner basis $G$  with respect to the lexicographical monomial ordering $y_1< y_2< \cdots < y_{n+m}$
            \smallskip
            \item [\rm 5:] If the basis $G$ has the following form
\[G=[w,y_2-v_2,\dots,y_{n+m}-v_{n+m}],\]
            where $w$ and $v_j$'s in $\R[y_1]$,  else $j:= j+1$. Denote \[v(t):=(t,v_2(t),\dots,v_{n+m}(t))\]
        \end{itemize}

        \item [\rm 6:] Obtain $\tilde{H}(t)$ and $\tilde{C}(t)$, respectively, by substituting $(x,\lambda)=A^{-1}v(t)$ in $\nabla_{x}^2 L$ and $x=\pi_n(A^{-1}v(t))$ in $\nabla h(x)$, then compute \[H(t):=\tilde{H}(t) \ \text{mod} \ w(t) \ \text{ and } \ C(t):=\tilde{C}(t) \ \text{mod} \ w(t)\]

        \item [\rm 7:] Return the set
\[\loc(f,V_{\R}(h)) = \Big\{\pi_n \left( A^{-1}\big(v(t)\big)\right) : w(t)=0, t\in\R, H(t)\succ_{C(t)^{\perp}} 0 \Big\}\]

    \end{itemize}
\end{algorithm}

The correctness of $\gralom$ is stated and proved briefly in the following proposition.

\begin{proposition}\label{thm:corect2} Let $f,h_1,\dots,h_m$ be in $\R[x]$ such that the conditions in Theorem \ref{thm:local2} hold. On input $f,h_1,\dots,h_m$, Algorithm $\gralom$
    terminates and returns all local minimizers of $f$ over $V_{\R}(h)$.
\end{proposition}

\begin{proof}
Since $\I_{\nabla}(L)$ is zero-dimensi\-onal, thanks to \cite[Lemma 2.1]{rouillier1999}, the while loop ends and then there exists an invertible matrix $A$ in $\Mat((n+m)\times (n+m),\R)$ such that $\rad(\I_{\nabla}(\ell))$ is in shape position.
After substituting at Line 6, we obtain matrices $H(t)$ and $C(t)$. The values of $C_{ij}(t)$  and $\tilde{C}_{ij}(t)$ are the same on the roots of $w(t)$. According to Theorem \ref{thm:local2},
\begin{center}
    $\loc(f,V_{\R}(h)) = \Big\{\pi_n \left( A^{-1}\big(v(t)\big)\right) : w(t)=0, t\in\R, H(t)\succ_{C(t)^{\perp}} 0 \Big\}$
\end{center}
is the set of local minimizers of $f$ over $V_{\R}(h)$.
\end{proof}


\subsubsection{Algorithm $\gralomplus$}
Let $f,h_1,\dots,h_m$ be satisfied the conditions in Theorem~\ref{thm:global2}. Suppose that the infimum of $f$ over $V_{\R}(h)$ attains. To compute the infimum value and set of global minimizers, we can design $\gralomplus$ whose tasks are similar to that ones of $\gralom$ and $\grulomplus$. Here, we skip its description and note that its correctness is proved based on Theorem \ref{thm:global2}.

\subsection{An experimental example}\label{ex:ineq}
To illustrate $\gralom$ in details, we compute local minimizers of the $2$-variable Rosenbrock function
$f=100x_1^4-200x_1^2x_2+x_1^2+100x_2^2-2x_1+1$ over the unit circle $\mathbb{S}^1=\{x\in\R^2: h(x)=0\}$, where $h=x_1^2+x_2^2-1$.
The gradient of $h$ is
$\nabla h = \left[2x_1 \, 2x_2\right]^T$ and its
rank is one on the real roots of $h$.

The Lagrangian function is the polynomial
$L(x,\lambda)=f(x)+\lambda (x_1^2+x_2^2-1)$ in three variables $x_1,x_2$ and $\lambda$. One can check that $\I_{\nabla}(L)$ is radical and zero-dimensional.
With the lexicographical order $x_1<x_2<\lambda$,  $\I_{\nabla}(L)$ is in shape position with the reduced Gr\"{o}bner basis $\I_{\nabla}(L)= \left\langle w(x_1),x_2-v_2(x_1), \lambda-v_{3}(x_1) \right\rangle$,
where

{\small $w \, = \, 40000x_1^8  + 10400 x_1^6  - 400 x_1^5  - 70599 x_1^4  + 598 x_1^3  + 30200 x_1^2  - 198 x_1 - 1$},
\smallskip

$v_2=-\frac{4243600}{21191}x_1^7 + \frac{123600}{21191}x_1^6
    - \frac{1106936}{21191}x_1^5+ \frac{72100}{21191}x_1^4
    + \frac{748774791}{2119100}x_1^3 \medskip + \frac{9251509}{2119100}x_1^2
    - $ \begin{center}
        $\frac{320081191}{2119100}x_1 + \frac{1048491}{2119100}$,
    \end{center}

    $v_3=-\frac{1080000}{21191}x_1^7 - \frac{24720000}{21191}x_1^6
    - \frac{1000800}{21191}x_1^5 - \frac{22896400}{21191}x_1^4
    + \frac{1486173}{21191}x_1^3 \medskip + \frac{26937036}{21191}x_1^2$
    \begin{center}
        $-\frac{194182}{2119100}x_1 - \frac{2120027}{21191}$.
    \end{center}
One defines $A=I_3$. The Hessian matrix of $L$ with respect to $x$ is
$\nabla^{2}_{x} L = \nabla^{2} f + 2 \lambda I_2$, with
\[\nabla^{2} f =
\begin{bmatrix}
    1200x_1^2-400x_1+2 \qquad & \qquad -400x_1 \\
    -400x_1  \qquad  &  \qquad 200
\end{bmatrix}.\]
In $\nabla^{2}_{x} L$, by substituting $(x,\lambda)=\left( t,v_2(t), v_3(t)\right)$, we obtain matrix $H(t)$ at Line 6 in $\gralom$. The subspace $C(t)^{\perp}$ associated to matrix $C(t)$ at Line 4 is given by
\[C(t)^{\perp}=\{u\in\R^2:tu_1+v_2(t)u_2=0\}.\]
From the above results, we have a univariate representation for the set of local minimizers of $f$ over $\Sp^1$:
\[\Big\{(t,v_2(t)): w(t)=0, t\in\R, H(t)\succ_{C(t)^{\perp}} 0\Big\}.\]
By isolating $w$, we obtain $6$ real roots:
\begin{center}
    {\small \begin{tabular}{ccc}
            -0.8684745451  & -0.7839301862 & -0.0033445316 \\
            \ 0.0099009901 & \ 0.7864151542  & \ 0.8658463102
    \end{tabular}}
\end{center}
There are three points at which $H(t)\succ_{C(t)^{\perp}} 0$. This allows us to conclude that $f$ has three local minimizers, denoted by $A,B,C$ whose coordinates are listed in Table 2.
\\

\begin{center}
    {\small
        \begin{tabular}{|c|c|c|c|}
            \hline
            point & $\alpha_1$         & $\alpha_2$         & $f(\alpha_1,\alpha_2)$      \\ \hline
            $A$      & -0.7839301862 & \ 0.6208489858  & 3.186378996 \\ \hline
            $B$      & \ 0.0099009901  & -0.9999509840 & 100.9900990 \\ \hline
            $C$      & \ 0.7864151542  & \ 0.6176983125  & 0.045674808 \\ \hline
        \end{tabular}
    } \\ \bigskip

    {\small \textbf{Table 2.} Local minimizers of the 2-variable Rosenbrock function over $\mathbb{S}^1$}\label{table:2}
\end{center}

\section{Squared slack variables technique handling inequality constraints}\label{sec:ineq}
In this section, we aim at computing local minimizers of problems having inequality constraints  by using the squared slack variables technique,
see for example \cite{tapia1974stable}.

We consider a polynomial optimization problem with inequality constraints as follows:
\begin{equation}\label{prob1}
    \min_{x\in\R^n} \big\{f(x):h(x)\geq 0\big\},
\end{equation}
where $h=(h_1,\dots,h_m)$. One introduces slack variables $z=(z_1,\dots,z_m)$
and consider a new polynomial optimi\-zation problem
\begin{equation}\label{prob2}
    \min_{(x,z)\in\R^{n \times m}} \big\{F(x,z):h_1(x)-z_1^2= 0, \dots, h_m(x)-z_m^2=0\big\},
\end{equation}
where
$F(x,z)=f(x)$. We denote by $V_{\R}(h-z^2)$ the constraint set of \eqref{prob2}.
Clearly,  $\alpha$ in $\R^n$ is a local (global) minimizer of \eqref{prob1} if and only if $(\alpha,\beta)$ in $\R^{n+m}$, where $\beta_1=\pm\sqrt{h_1(\alpha)}$, \dots, $\beta_m=\pm\sqrt{h_m(\alpha)}$, is a  local (global) minimizer of \eqref{prob2}. Thus, the set of all local minimizers of \eqref{prob1} is given by
    \begin{equation}\label{eq:locfK}
        \loc(f,[h\geq 0]) =
        \pi_n\Big(\loc\big(F,V_{\R}(h-z^2)\big)\Big).
    \end{equation}

\begin{remark}\label{rmk:p1p2} To compute local (global) minimizers of the original problem \eqref{prob1} we perform the two following steps: First, we compute that ones of the reformulated problem \eqref{prob2} under genericity conditions by using the technique in Section \ref{sec:eq}; After that, we apply \eqref{eq:locfK}. When such conditions hold for the reformulated problem then the original problem satisfies the regularity condition and the second-order sufficient condition studied in \cite{fukuda2017note} and it is clear that the original problem has finitely many local minimizers.
\end{remark}

\begin{example} We illustrate Remark \ref{rmk:p1p2} by considering
the $2$-variable Rosenbrock function
$f$, given in Section \ref{ex:ineq}, over the unit disk $\mathbb{D}^1$ in $\R^2$, with $h(x_1,x_2)=1-x_1^2-x_2^2$.
We introduce slack variable $z$
and consider a new problem whose objective function is $F(x_1,x_2,z)=f(x_1,x_2)$ and constraint set is given by
\[V_{\R}(h-z^2)=\{(x_1,x_2,z)\in{\R}^{3}:1-x_1^2-x_2^2-z^2= 0\}.\]
The gradient of $h-z^2$ is
$\left[-2x_1 \, -2x_2, -2z\right]^T$ and its
rank is one on the real roots of $h-z^2$.
By performing $\gralom$, we obtain the set of local minimizers of $F$ over $V_{\R}(h-z^2)$ that contains only one point as follows:
\[\loc\big(F,V_{\R}(h-z^2)\big) = \big\{ \big(0.786415154, 0.617698312, 0\big) \big\}.\]
From \eqref{eq:locfK}, the set of local minimizers of $f$ over $\mathbb{D}^1$ is the following:
\[\loc\big(f,\mathbb{D}^1 \big) = \big\{ \big(0.786415154, 0.617698312\big) \big\}.\]
\end{example}

\section{Perturbation technique handling positive dimension cases}\label{sec:posdim}
Comput\-ing local and global mini\-mizers of a polynomial optimization problem under the condition that the gradient variety of $L(x,\lambda)$ has finitely many points was discussed in Section \ref{sec:eq}. We now study
the case where the gradient variety might have infinitely many points. In particular, based on a perturbation technique and our algorithm $\gralomplus$,
we are able to compute approxim\-ately a global minimizer and the minimum of an optimization problem provided that its constraint set is compact.


Let $f,h_1,\dots,h_m$ be $m+1$ polynomials in $\R[x]$.  We consider perturbed problems of minimizing
\[f_{\eps}(x):=f(x)+ \eps^T x \]
over $V_{\R}(h)$, for $\eps$ in $\R^n$. Denote by $L_{\eps}$ the Lagrangian function associated to $f_{\eps}$.

In general, accumulation points of local minimizers $\loc(f_{\eps},V_{\R}(h))$ may not be local minimizers of the original problem as $\eps \to 0$. For example, one considers the univariate polynomial $f(t)=-t^3/3$ over $\R$; Clearly, $t_{\eps}=\sqrt{-\eps}$ is a local minimizer of  $f_{\eps}(t)=t^3/3+\eps t$ but $\lim_{\eps\to 0^{-}}t_{\eps} = 0$ is not a local minimizer of $f$.
However, for global minimizers, in this section we point out that the perturbation technique works.

We now prove the following lemma related to genericity of being zero-dimensional of $\I_{\nabla}(L_{\eps})$ relying on Sard Theorem with parameter \cite[Theorem 2.4]{dang2016well}.

\begin{lemma}\label{lm:epsgen} If $\nabla h$ is full rank over the complex variety $V_{\C}(h)$, then the ideal $\I_{\nabla}(L_{\eps})$ is zero-dimensional for generic $\eps$.
\end{lemma}

\begin{proof} The variety $V_{\C}(h)$ is smooth since $\nabla h$ is full rank over $V_{\C}(h)$. Consider the smooth map $\Phi:\R^n\times V_{\C}(h)\times \C^{m} \to \C^{n}\times\C^{m}$ defined by
    \[\Phi(\eps,x,\lambda) = \left(\frac{\partial L_{\eps}}{\partial x_1},\dots, \frac{\partial L_{\eps}}{\partial x_n}, h_1(x),\dots,h_m(x) \right).\]
    Clearly, its Jacobian matrix with respect to $(\eps,x)$ is given by
    \[\begin{bmatrix}
        I_n & 0 \\
        0 & \nabla h^T
    \end{bmatrix}.\]
    This matrix is full rank for all $x$ in $V_{\C}(h)$. It follows that the origin $0$ in $\C^{n}\times\C^{m}$ is a regular value of $\Phi$. By Sard Theorem with parameter \cite[Theorem 2.4]{dang2016well}, there exists an open and dense set $\mathcal{O}$ in $\R^n$ such that for each $\eps$ in $\mathcal{O}$, $0$ in $\C^{n}\times\C^{m}$ is a regular value of $\Phi_{\eps}: V_{\C}(h)\times\C^{m} \to \C^{n}\times\C^{m}$ defined by $\Phi_{\eps}(x,\lambda)=\Phi(\eps,x,\lambda)$.
    It is clear that $\Phi_{\eps}^{-1}(0)=V_{\C}(\I_{\nabla}(L_{\eps}))$ and if this set is non-empty then it  has finitely many points.
\end{proof}

By applying \cite[Theorem 5.2]{li2015new} to our perturbed problems, we obtain following result about stability of global minimizers:

\begin{proposition}\label{lm:holder} Suppose that $V_{\R}(h)$ is compact. Then, there are positive constants $c,\tau,\nu$ such that
    \[\glo(f_{\eps},V_{\R}(h)) \subset \glo(f,V_{\R}(h)) + c\|\eps \|^{\tau} \overline{\mathbb{B}}(0,1)\]
    whenever $\|\eps\|\leq \nu$.
\end{proposition}

\begin{proof} Compactness of $V_{\R}(h)$ implies that there is some $\gamma>0$ such that $\|x\|\leq \gamma$ for all $x$ in $V_{\R}(h)$. Therefore,  one has
\[|f_{\eps}(x) - f(x)| \leq \gamma \|\eps\|\]
    for all $x$ in $V_{\R}(h)$ and for all $\eps$. This implies that Assumption 2 in \cite[Section 5.1]{li2015new} holds. By applying  \cite[Theorem 5.2]{li2015new}, we obtain the conclusion in the lemma.
\end{proof}

\begin{remark}\label{rmk:posdim}
    Suppose that $\nabla h$ is full rank over $V_{\C}(h)$ and $V_{\R}(h)$ is compact. One sees that, if parameter $\eps$ is chosen generically then, according to assertion (ii) in \cite[Theorem A]{lee2017generic}, one has $\#\glo(f_{\eps},V_{\R}(h))=1$, denoting by $\alpha_{\eps}$ the global minimizer. Moreover, from Proposition~\ref{lm:holder}, if $\alpha_{\eps} \to \alpha$ as $\eps\to 0$ then $\alpha$ is a global minimizer of the original optimization problem, i.e. $\alpha$ belongs to $ \glo(f,V_{\R}(h))$. We would remark that, thanks to Lemma \ref{lm:epsgen}, $\I_{\nabla}(L_{\eps})$ is zero-dimensional, hence $\alpha_{\eps}$ can be computed by using our algorithm $\gralomplus$ in Section \ref{sec:eq}.
\end{remark}

\begin{example} We demonstrate Remark \ref{rmk:posdim} by considering
    $f(x_1,x_2,x_3)=x_1^4$ over the unit circle $\mathbb{S}^1$ in $\R^3$, with $h(x_1,x_2)=x_1^2+x_2^2+x_3^1-1$. The gradient of $h$ has
    rank one on the roots of $h$. It is easy to see that the minimum value is $0$ and the set of global minimizer of $f$ over $\mathbb{S}^1$ is as follows:
    \[\glo(f,\mathbb{S}^1)=\big\{(0,\alpha_2,\alpha_3): \alpha_2^2+\alpha_3^2=1\big\}\]
that has dimension one, hence $V_{\C}(\I_{\nabla}(L))$ has positive dimension.
We now consider linear perturbation of $f$, for $(\eps_1,\eps_2,\eps_3)$ in $\R^3$, as follows:
\begin{center}
    $f_{\eps}(x):=x_1^4+\eps_1 x_1 + \eps_2 x_2+ \eps_3 x_3$.
\end{center}
We take $\eps_1=\eps_2=\eps_3$ randomly, compute a global minimizer $\alpha_{\eps}=(\alpha_{\eps,1},\alpha_{\eps,2},\alpha_{\eps,3})$  of $f_{\eps}$ over $\mathbb{S}^1$ by applying $\gralomplus$ with setting the variable $\texttt{Digits}= 10$, and then obtain the following table:

\begin{center}
{\small \begin{tabular}{|c|c|c|c|c|}
        \hline
        $\eps_1=\eps_2=\eps_3$ & $\alpha_{\eps,1}$ & $\alpha_{\eps,2}$ & $\alpha_{\eps,3}$                    & $f_{\eps,\min}$ \\ \hline
        $10^{-5}$              & $-0.01348524792$  & $-0.7070425062$   & $-0.7070425062$                      & $-0.000014242$  \\ 
        $10^{-7}$              & $-0.00291998727$  & $-0.7071036821$   & $-0.7071036821$                      & $-0.000000141$  \\ 
        $10^{-9}$              & $-0.00062977344$  & $-0.7071066611$   & $-0.7071066611$ & $-0.000000001$
        \\ 
        $ \ 10^{-11}$             & $-0.00013571219$  & $-0.7071048281$   & $-0.7071048281$ & $-0.000000000$  \\ \hline
\end{tabular}

{\small \textbf{Table 3.} Global minimizers of function $f$ over $\mathbb{S}^1$}\label{table:3}
}
\end{center}
From the results in Table 3 we see that, when the parameters $\eps_1=\eps_2=\eps_3$ approach $0$,  the perturbed minimum $f_{\eps,\min}$ goes to $0$ and   $\alpha_{\eps}$, with $\alpha_{\eps,2}=\alpha_{\eps,3}$, approaches $(0,-\sqrt{2},-\sqrt{2})$ which is a global minimizer of $f$ over $\mathbb{S}^1$.
\end{example}

\bibliographystyle{spmpsci}
\bibliography{references.bib}

\end{document}